\documentclass[10pt]{amsart}
\usepackage{fancyhdr}
\usepackage{stmaryrd}
\usepackage{calrsfs}

\usepackage{amsmath}
\usepackage[nospace,noadjust]{cite}
\usepackage{amsfonts}
\usepackage{amssymb,enumerate}
\usepackage{amsthm}
\usepackage{cite}
\usepackage{comment}
\usepackage{color}
\usepackage[all]{xy}
\usepackage{hyperref}
\usepackage{lineno}
\usepackage{tikz-cd}

\newtheorem{theorem}{Theorem}[section]
\newtheorem{prop}[theorem]{Proposition}
\newtheorem{question}[theorem]{Question}

\newtheorem{lemma}[theorem]{Lemma}
\newtheorem{cor}[theorem]{Corollary}
\newtheorem*{main-theorem}{Main Theorem}

\theoremstyle{definition}

\newtheorem{example}[theorem]{Example}
\numberwithin{equation}{section}

\newcommand{\ff}{\mathbb{F}}
\newcommand{\nn}{\mathbb{N}}
\newcommand{\pp}{\mathbb{P}}
\newcommand{\qq}{\mathbb{Q}}
\newcommand{\rr}{\mathbb{R}}
\newcommand{\zz}{\mathbb{Z}}

\newcommand{\uu}{\mathcal{U}}

\providecommand\ldb{\llbracket}
\providecommand\rdb{\rrbracket}

\hypersetup{
	pdftitle={FD},
	colorlinks=true,
	linkcolor=blue,
	citecolor=cyan,
	urlcolor=wine
}

\keywords{monoid domain, monoid semidomain, almost atomicity, quasi-atomicity, atomic domain, integral domain, semiring}

\subjclass[2020]{Primary: 13F15, 13A05, 20M25; Secondary: 06F05, 11Y05, 13G05}

\begin{document}
	
	\mbox{}
	\title{On the ascent of almost and quasi-atomicity \\ to monoid semidomains}
	
	\author{Victor Gonzalez}
	\address{Department of Mathematics\\Miami Dade College\\Miami, FL 33135}
	\email{vmichelg@mdc.edu}
	
	\author{Felix Gotti}
	\address{Department of Mathematics\\MIT\\Cambridge, MA 02139}
	\email{fgotti@mit.edu}
	
	\author{Ishan Panpaliya}
	\address{Department of Mathematics\\Seattle University\\Seattle, WA 98105}
	\email{panpaliya.ishan@gmail.com}
	
\date{\today}

\begin{abstract}	 
	 A commutative monoid is atomic if every non-invertible element factors into irreducibles (also called atoms), while an integral (semi)domain is atomic if its multiplicative monoid is atomic. Notions weaker than atomicity have been introduced and studied during the past decade, including almost atomicity and quasi-atomicity, which were coined and first investigated by Boynton and Coykendall in their study of graphs of divisibility of integral domains. The ascent of atomicity to polynomial extensions was settled by Roitman back in 1993 while the ascent of atomicity to monoid domains was settled by Coykendall and the second author in 2019 (in both cases the answer was negative). The main purpose of this paper is to study the ascent of almost atomicity and quasi-atomicity to polynomial extensions and monoid domains. Under certain reasonable conditions, we establish the ascent of both properties to polynomial extensions (over semidomains). Then we construct an explicit example illustrating that, with no extra conditions, quasi-atomicity does not ascend to polynomial extensions. Finally, we show that, in general, neither almost atomicity nor quasi-atomicity ascend to monoid domains, improving upon a construction first provided by Coykendall and the second author for the non-ascent of atomicity.
\end{abstract}
\medskip

\maketitle

%\tableofcontents

\bigskip
%%%%%%%%%%%
%%%%%%%%%%%
\section{Introduction}
\label{sec:intro}
\smallskip

It is well known that the unique factorization property (i.e., the statement of the Fundamental Theorem of Arithmetic (FTA)) ascends from any integral domain to its polynomial extension, that is, if an integral domain $R$ is a UFD so is $R[x]$. The first version of this result, often referred to as Gauss's lemma, was established by Gauss~\cite{cfG1801} back in 1801, and it is fundamental in the context of commutative ring theory. As the property of unique factorization, the property of being Noetherian (i.e., satisfying the ACC on ideals) ascends to polynomial extensions, and this fact is known as Hilbert's basis theorem and was proved by Hilbert~\cite{dH1890} back in 1890. In general, establishing the ascent of certain algebraic/arithmetic property to polynomial extensions allows us to identify a new class of commutative rings satisfying the given property. This was a fundamental part of Noether's classical approach to commutative ring theory.
\smallskip

The study of factorizations goes back to the nineteenth century as the same was crucial to better understand and approach fundamental problems in number theory, and it was essential in the development of what we call today algebraic number theory and commutative ring theory. One of the first relevant works on factorizations is present in Kummer's development of the theory of ideal numbers, which was motivated by the failure of certain cyclotomic extensions $\zz[\zeta_p]$ (here $\zeta_p$ is a $p$-th root of unity) to be UFDs \cite{eK1847,eK1851}. Based on Lam\'e's approach (which wrongly assumes that $\zz[\zeta_p]$ is a UFD for each prime~$p$), Kummer was able to use his theory of ideal numbers to prove the statement of Fermat's Last Theorem for a significantly large class of prime exponents. This was an essential step in the development of algebraic number theory. In addition, Kummer's study of ideal numbers significantly influenced Dedekind's pioneering work on ideal theory, turning the notion of ideal numbers into the more sophisticated notion of ideals (as we know them today) and giving initial shape to the formal algebraic machinery we currently know as commutative ring theory.
\smallskip

However, it was not until 1990 that factorizations were studied in the more general setting of integral domains. This study was initiated in~\cite{AAZ90}, where Anderson, Anderson, and Zafrullah introduced the classes of bounded factorization domains (BFD) and finite factorization domains (FFD). The notion of an FFD is a natural generalization of that of a UFD, while the notion of a BFD is a natural generalization of that of an FFD. In the same paper, the authors studied the classes of BFDs and FFDs in connection with fundamental number-theoretical notions, including the IDF property (i.e., the existence of irreducible divisors) and the property of atomicity (i.e., the existential condition in the statement of the FTA) and also in connection with some classic ring-theoretical properties, including the Krull/Noetherian properties and the ACCP (i.e., the ACC on principal ideals). In~\cite{AAZ90}, the authors not only proved that both the bounded/finite factorization property ascends to polynomial extensions but also highlighted the ascent of atomicity \cite[Question~1]{AAZ90} as well as the ascent of the IDF property \cite[Question~2]{AAZ90} as open questions in their paper (the ascent of the Krull property to polynomial extensions had already been proved in~\cite[Theorem~1]{lC81}, while the ascent of the ACCP can be easily verified). Recent surveys on factorizations can be found in~\cite{AG22} for the classical setting of integral domains and in~\cite{GZ20} for the more general setting of commutative monoids.
\smallskip

Although we already mentioned that the property of atomicity refers to the existential statement in the FTA, a more formal and general definition goes as follows: a cancellative and commutative monoid is called atomic if every non-invertible element factors into finitely many atoms (i.e., irreducible elements), while an integral domain is atomic if its multiplicative monoid is atomic. Atomicity has been largely studied in connection with the ACCP (see \cite{pC68,aG74} and the more recent papers \cite{CGH23,GL23}). The ascent of atomicity was first pointed out as an open problem by Gilmer~\cite[page 189]{rG84} (he, along with Parker, had already proved that the ACCP does not ascend, in general, from pairs $(M,R)$ consisting of a torsion-free monoid and an integral domain to their monoid domain $R[M]$ \cite[Corollary~7.14]{GP74}). The ascent of atomicity to polynomial rings is the special version of this question when $M$ is a rank-one free commutative monoid (i.e., an isomorphic copy of $\nn_0$). The ascent of atomicity to polynomial rings, proposed in~\cite[Question~1]{AAZ90}, was answered by Roitman~\cite{mR93}, who provided a technical but systematic construction of atomic domains having certain desirable algebraic/divisibility properties, which can be manipulated to achieve non-atomicity in their corresponding polynomial extensions (the non-ascent of atomicity to power series ring is also due to Roitman~\cite{mR00}). A recent survey on atomicity in the setting of integral domains can be found in~\cite{CG24}.
\smallskip

The main purpose of this paper is to study the potential ascent of almost and quasi-atomicity to both polynomial semidomains and monoid semidomains. Almost and quasi-atomicity are two natural generalizations of atomicity introduced by Boynton and Coykendall in 2015 in their study of graphs of divisibility of integral domains. In a cancellative and commutative monoid $M$, we say that an element is atomic if it is invertible or factors into finitely many atoms, and then we say that the monoid $M$ is almost atomic (resp., quasi-atomic) if for each $b \in M$ there exists an atomic element $a \in M$ (resp., an element $a \in M$) such that $ab$ is atomic in~$M$. Following Boynton and Coykendall~\cite{BC15}, we say that an integral domain is almost atomic (resp., quasi-atomic) provided that its multiplicative monoid is almost atomic (resp., quasi-atomic). The properties of almost atomicity and quasi-atomicity have also been considered by Lebowitz-Lockard~\cite{nL19}, where he provided various interesting constructions of non-atomic domains satisfying at least one of such properties.
\smallskip

In Section~\ref{sec:background}, in an attempt to make this paper as self-contained as possible, we briefly revise the terminology, notation, and main known results that we will use throughout the rest of the paper.
\smallskip

In Section~\ref{sec:monoid semidomains}, we study the ascent of atomicity in the general setting of monoid semidomains. First, we argue that, for any monoid $M$ that satisfies the ACCP, if a semidomain $S$ is almost atomic (resp., quasi-atomic), then the monoid semidomain $S[M]$ is also almost atomic (resp., quasi-atomic). The special case $M = \nn_0$ was addressed by Polo and the second author in~\cite{GP23}, where they proved that the existence of certain maximal common divisors in $S$ is a sufficient condition for the ascent of both almost atomicity \cite[Theorem~4.2]{GP23} and quasi-atomicity \cite[Theorem~5.3]{GP23}. The same condition was first proposed and proved to be sufficient by Roitman for the ascent of atomicity to polynomial extensions~\cite{mR93}. The second part of this section is devoted to enlarge the classes of known almost atomic domains and quasi-atomic domains by considering ring extensions of the form $S[x] + F[x]x^2$, where $S$ is an integral domain and $F$ is a field containing $S$. Our work is motivated by the prototypical examples $\zz[x] + \qq[x]x^2$ and $\zz[x] + \rr[x]x^2$ Lebowitz-Lockard discussed in~\cite[Section~2]{nL19} as an example of an almost atomic domain that is not atomic and a quasi-atomic domain that is not almost atomic, respectively.
\smallskip

In Section~\ref{sec:polynomial rings}, our main purpose is to establish the non-ascent of quasi-atomicity to polynomial extensions, which is a result parallel to the non-ascent of atomicity to polynomial extensions already provided in~\cite{mR93}. As mentioned in the previous paragraph, the subring $R := \zz[x] + \rr[x]x^2$ of $\rr[x]$ is a quasi-atomic domain that is not almost atomic. Although $R$ is a quasi-atomic domain, we will prove that its polynomial extension $R[x]$ is not quasi-atomic, obtaining that the property of being quasi-atomic does not ascend to polynomial extensions. We should emphasize that we were unable to answer whether the property of almost atomicity ascends to polynomial extensions and, in this direction, we emphasize the corresponding open question to motivate readers to complement or extend the results we have obtained here.
\smallskip

In Section~\ref{sec:monoid domains}, we prove that for any prime $p$ neither the property of almost atomicity nor that of quasi-atomicity ascend from monoids to their corresponding monoid domains over the finite field~$\ff_p$. In the recent paper~\cite{GR25}, Rabinovitz and the second author constructed a rank-one torsion-free atomic monoid $M$ such that for any integral domain $R$ the monoid domain $R[M]$ is not atomic. The first result in this direction, establishing the non-ascent of atomicity to monoid domains, was given by Coykendall and the second author in~\cite{CG19}: for each prime $p$, they constructed a finite-rank (indeed, rank at most~$2$) torsion-free atomic monoid $M$ whose monoid domain $\ff_p[M]$ is not atomic. We conclude this paper improving and generalizing the same result. The extent to which we generalize this result allows us to provide a negative answer to the ascent of almost atomicity and quasi-atomicity to monoid domains over fields. Although we reuse crucial ideas introduced in \cite[Section~5]{CG19}, the construction we present here contains significant deviations from some of the arguments presented in~\cite{CG19}, and we hope such deviations are useful to tackle the ascent of further algebraic/arithmetic properties to monoid algebras.

\bigskip
%%%%%%%%%%%%%%%%%%%%
%%%%%%%%%%%%%%%%%%%%
\section{Background}
\label{sec:background}

\smallskip
%%%%%%%%%%%%%%%%
\subsection{General Notation}
\smallskip

Following usual conventions, we let $\zz$, $\qq$, and $\rr$ denote the set of integers, rational numbers, and real numbers, respectively. In addition, we let $\pp$, $\nn$, and $\nn_0$ denote the set of primes, positive integers, and nonnegative integers, respectively. Then, for each $p \in \pp$, we let $v_p \colon \qq^\times \to \zz$ denote the standard $p$-adic valuation map. For $b,c \in \zz$ with $b \le c$, we let $\ldb b, c \rdb$ stand for the discrete interval from $b$ to $c$:
\[
	\ldb b,c \rdb := \{n \in \zz : b \le n \le c\}.
\]
Also, for $S \subseteq \rr$ and $r \in \rr$, we set $S_{\ge r} := \{s \in S : s \ge r\}$, and we define $S_{> r}$, $S_{\le r}$ and $S_{< r}$ similarly. For each $q \in \qq_{>0}$, we let $\mathsf{n}(q)$ and $\mathsf{d}(q)$ be the unique relatively prime positive integers $n$ and $d$ such that the equality $q = \frac nd$ holds.

\medskip
%%%%%%%%%%%%%%%%%%%
\subsection{Commutative Monoids}

Although a monoid is usually defined as a semigroup with an identity element, in the scope of this paper, the term \emph{monoid} refers to a cancellative and commutative semigroup with an identity element. Let $M$ be a monoid (multiplicatively written). The group consisting of all invertible elements of $M$ is denoted by $\uu(M)$ and the elements in $\uu(M)$ are called \emph{units}. When $\uu(M)$ is the trivial group, we say that $M$ is \emph{reduced}. The \emph{Grothendieck group} of $M$, denoted by $\mathcal{G}(M)$, is the abelian group (unique up to isomorphism) that satisfies the following property: if an abelian group contains an isomorphic copy of $M$, then it must contain an isomorphic copy of $\mathcal{G}(M)$. The rank of a monoid is the rank of its Grothendieck group as a $\zz$-module. A monoid is called \emph{torsion-free} provided that its Grothendieck group is torsion-free. Every rank-$1$ torsion-free abelian group can be realized as an additive subgroup of $\qq$ \cite[Section~24]{lF70}, and so it follows from \cite[Theorem~2.9]{rG84} that every rank-$1$ torsion-free monoid can be realized as an additive submonoid of $\qq_{\ge 0}$.
\smallskip

An element $a \in M \! \setminus \! \uu(M)$ is called an \emph{atom} (or an \emph{irreducible element}) if, for all $b,c \in M$, the equality $a = bc$ implies that $\{b,c\}$ intersects $\uu(M)$. We let $\mathcal{A}(M)$ denote the set consisting of all the atoms of $M$. More generally, an element $a \in M$ is called \emph{atomic} if either $a \in \uu(M)$ or $a$ factors into finitely many atoms. Following Cohn~\cite{pC68}, we say that $M$ is \emph{atomic} if every element of $M$ is atomic. We say that an element $b \in M$ is \emph{almost atomic} (resp., \emph{quasi-atomic}) if we can take an atomic element (resp., an element) $a \in M$ such that $ab$ is atomic. Clearly, an atomic element is almost atomic, whereas an almost atomic element is quasi-atomic. Following Boynton and Coykendall~\cite{BC15}, we say that a monoid is \emph{almost atomic} (resp., \emph{quasi-atomic}) if all its elements are almost atomic (resp., quasi-atomic). Almost atomicity and quasi-atomicity are the most relevant algebraic notions in the scope of this paper.
\smallskip

A subset $I$ of $M$ is called an \emph{ideal} if $IM := \{bc : b \in I \text{ and } c \in M\}$ is a subset of $I$, and we say that an ideal $I$ is \emph{principal} if $I = bM := \{bm : m \in M\}$ for some $b \in M$. An element $b_0 \in M$ satisfies the \emph{ascending chain condition on principal ideals} (ACCP) if every ascending chain of principal ideals \emph{starting} at $b_0$ stabilizes, which means that if $(b_nM)_{n \ge 0}$ is an ascending chain consisting of principal ideals of $M$, then there exists $n_0 \in \nn$ such that $b_n M = b_{n_0} M$ when $n \ge n_0$. Accordingly, we say that~$M$ satisfies the \emph{ACCP} if every element of $M$ satisfies the ACCP. If a monoid satisfies the ACCP, then one can readily check that it is atomic. For any $b,c \in M$, we say that $b$ \emph{divides} $c$ and write $b \mid_M c$ if $c \in bM$ or, equivalently, $cM \subseteq bM$. A submonoid $N$ of $M$ is called \emph{divisor-closed} if the only pairs $(b,c) \in M \times N$ with $b \mid_M c$ are those in $N \times N$. Let $S$ be a nonempty subset of $M$. A \emph{common divisor} of~$S$ is an element $d \in M$ such that $d \mid_M s$ for all $s \in S$, while a common divisor $d \in M$ of $S$ is a \emph{maximal common divisor} (MCD) if the only common divisors of the set $\big\{\frac{s}{d} : s \in S \big\}$ are the units of~$M$. The monoid $M$ is called an \emph{MCD monoid} if each nonempty finite subset of~$M$ has an MCD. 
%It is clear that every greatest common divisor is a maximal common divisor, and so every GCD monoid is an MCD monoid. For $k \in \nn$, we say that $M$ is a $k$-\emph{MCD monoid} if every subset of $M$ of cardinality $k$ has a maximal common divisor. Clearly, every monoid is a $1$-MCD monoid. Also, observe that $M$ is an MCD monoid if and only if $M$ is a $k$-MCD monoid for every $k \in \nn$. The notion of $k$-MCD was introduced by Roitman in~\cite{mR93}. %Finally, we say that~$M$ is an \emph{MCD-finite} monoid if every nonempty finite subset of $M$ has only finitely many maximal common divisors up to associates. The notions of an MCD, $k$-MCD (for every $k \in \nn$), and MCD-finite are fundamental in the scope of this paper.
\smallskip

In this last paragraph, it is convenient to assume that $M$ is written additively, and we do so. We say that $M$ is \emph{linearly ordered} with respect to an order relation $\preceq$ on $M$ provided that for all $b,c,d \in M$ the order relation $b \prec c$ guarantees $b+d \prec c+d$. Then we say that $M$ is \emph{linearly orderable} if $M$ is a linearly ordered monoid with respect to some total order relation on~$M$. It is due to Levi~\cite{fL13} that every torsion-free abelian group is linearly orderable. Since monoids here are assumed to be cancellative and commutative, this implies that a monoid is linearly orderable if and only if it is torsion-free. Another consequence of the mentioned Levi's result is that if $M$ is reduced and torsion-free, then its Grothendieck group $\mathcal{G}(M)$ can be turn into a linearly ordered abelian group with respect to an order relation $\preceq$ such that $M$ is contained in the nonnegative cone of $\mathcal{G}(M)$, which means that $M \subseteq \{g \in \mathcal{G}(M) : 0 \preceq g \}$. An additive submonoid of $\qq$ consisting of nonnegative rational is called a \emph{Puiseux monoid}. The atomic structure and arithmetic of factorizations of Puiseux monoids and their monoid domains have been actively investigated for almost a decade (see~\cite{CJMM24,GG24,fG19,GL23} and references therein). Certain Puiseux monoids play a crucial role in the main construction of Section~\ref{sec:monoid domains}.
%For a subset $S$ of $M$, we let $\langle S \rangle$ denote the submonoid of $M$ \emph{generated} by~$S$, namely, the smallest submonoid of $M$ containing $S$. The monoid~$M$ is called \emph{finitely generated} if $M = \langle S \rangle$ for some finite subset $S$ of $M$.

\medskip
%%%%%%%%%%%%%%%%%%%%
\subsection{Commutative Semirings}

A set $S$ endowed with two binary operations `$+$' and~`$\cdot$', called \emph{addition} and \emph{multiplication}, respectively, is called a \emph{commutative semiring} if $(S,+)$ is a monoid, $(S, \cdot)$ is a commutative semigroup, and they are compatible via the distributive law: $(b+c) \cdot d = b \cdot d + c \cdot d$ for all $b, c, d \in S$. As every algebraic structure we deal with in the scope of this paper is commutative, from now on we refer to any commutative semiring simply as a semiring. The additive identity element of a semiring is denoted by $0$, and we tacitly assume from now on that any semiring we mention here has a multiplicative identity, which we denote by $1$. Let $S$ be a semiring. For any $b,c \in S$, we often write $b c$ instead of $b \cdot c$. A subset $S'$ of $S$ is called a \emph{subsemiring} of~$S$ provided that $(S',+)$ is a submonoid of $(S,+)$ that is closed under multiplication and contains~$1$. For semirings $S$ and $T$, a map $\phi \colon S \to T$ is called a \emph{semiring homomorphism} if $\phi$ is a monoid homomorphism from $(S,+)$ to $(T, +)$ such that $\phi(1) = 1$ and $\phi(bc) = \phi(b)\phi(c)$ for all $b,c \in S$. Semiring isomorphisms are defined in the obvious way.
\smallskip

Given that $(S,+)$ is cancellative and commutative, its Grothendieck group $\mathcal{G}(S)$ is an (additive) abelian group: $S$ is called a \emph{semidomain} if its multiplication can be extended to $\mathcal{G}(S)$ turning $\mathcal{G}(S)$ into an integral domain. It is clear that any integral domain is a semidomain, and it is not hard to show that~$S$ is a semidomain if and only if $S$ is isomorphic to a subsemiring of an integral domain \cite[Lemma~2.1]{GP23}. Let $S$ be a semidomain. The subset consisting of all the nonzero elements of $S$ is clearly a monoid under multiplication, which is called the \emph{multiplicative monoid} of $S$ and denoted by $S^*$. The \emph{group of units} of $S$, denoted by $S^\times$, is the group of units of~$S^*$. We say that an element of $S$ is an \emph{atom} (resp., an \emph{atomic element}) if it is an atom (resp., an atomic element) of $S^*$, and we let $\mathcal{A}(S)$ denote the set of atoms of $S$ (of course, $\mathcal{A}(S) = \mathcal{A}(S^*)$). We say that the semidomain $S$ is \emph{atomic} (resp., \emph{almost atomic}, \emph{quasi-atomic}) if its multiplicative monoid $S^*$ is atomic (resp., almost atomic, quasi-atomic). For $b,c \in S$, if $b$ divides $c$ in $S^*$, then we write $b \mid_S c$ instead of $b \mid_{S^*} c$. %Finally, for any nonempty subset $T$ of $S^*$, we write $\gcd_S(T)$ (resp., $\mcd_S(T)$) instead of $\gcd_{S^*}(T)$ (resp., $\mcd_{S^*}(T)$). As for the notion of units, throughout this paper we never consider divisibility in the additive monoid of any semidomain.
\smallskip

\medskip
%%%%%%%%%%%%%%%%%%
\subsection{Monoid Semidomains} 

Let $S$ be a semidomain, and let $M$ be a torsion-free monoid. The \emph{monoid semidomain} of~$M$ over $S$, denoted by $S[x;M]$ is defined to be the semiring consisting of all formal polynomial expressions with exponents in $M$ and coefficients in~$S$ (under polynomial-like addition and multiplication). We refer to elements of $S$ as \emph{constant} (\emph{polynomial expressions}). We write $S[M]$ instead of $S[x;M]$ when we see no risk of ambiguity. The most relevant special cases of monoid semidomains are the following:
\begin{itemize}
	\item $M = \nn_0$, in which case, $S[\nn_0]$ is denoted by $S[x]$ and called the \emph{polynomial semidomain} or the \emph{polynomial extension} of $S$, and
	\smallskip
	
	\item $S$ is an integral domain, in which case, $S[M]$ is called the \emph{monoid domain} of~$M$ over~$S$. 
\end{itemize}
Several results on the ascent of sub-atomic and arithmetic properties to polynomial semidomains were recently investigated in~\cite{GP23} and \cite{GP23a}, respectively (almost atomicity and quasi-atomicity are two of the sub-atomic properties considered in~\cite{GP23}). On the other hand, the ascent of atomic and factorization properties to monoid domains has been the subject of a great deal of study in recent years (see, for instance, \cite{CG19,GR25,BGLZ24,MO09}).
\smallskip

Let $R$ denote the Grothendieck group of $(S,+)$, and assume that $R$ is an integral domain containing~$S$ as a subsemiring (we can do so because $S$ is a semidomain). It is clear that $S[M]$ is a subsemiring of $R[M]$, which is an integral domain, as it is well known that the monoid ring of a torsion-free monoid over an integral domain is an integral domain \cite[Theorem~8.1]{rG84}. Therefore $S[M]$ is a semidomain. Also, $S[M]$ admits only trivial units as one can readily verify that
\begin{equation} \label{eq:group of units of a semidomain}
	S[M]^\times = \{ t x^u : t \in S^\times \text{ and } u \in \uu(M)\},
\end{equation}
which is well known in the special case of monoid domains \cite[Theorem~11.1]{rG84}. Because $S[M]$ is a subsemiring of $R[M]$, any nonzero polynomial expression $f(x)$ of $S[M]$ can be written in the following way:
\begin{equation} \label{eq:generic element in a monoid algebra}
	f(x) = c_1 x^{q_1} + \dots + c_n x^{q_n},
\end{equation}
where $c_1, \dots, c_n \in S$ are nonzero coefficients and $q_1, \dots, q_n \in M$ are pairwise distinct exponents. Then the \emph{support} of $f(x)$, denoted by $\text{supp} \, f(x)$, is the set of exponents $\{q_1, \dots, q_n\}$, which is uniquely determined by $f(x)$. We say that $f(x)$ is a \emph{monomial} if $\text{supp} \, f(x)$ is a singleton. 
\smallskip
%For any $q \in M$, we let $[q]f(x)$ denote the coefficient of $x^q$ in the representation of $f(x)$ given in~\ref{eq:generic element in a monoid algebra}: $[q]f(x) = 0$ if $q \notin \text{supp} \, f(x)$ while $[q_i]f(x) = c_i$ for every $i \in \ldb 1,k \rdb$. 
\smallskip

As $M$ is torsion-free, it must be linearly orderable. Fix an order relation~$\preceq$ on $M$ turning~$M$ into a linearly ordered monoid. Then the support of $f(x)$ is also linearly ordered, and so we can assume that $q_1 \prec \dots \prec q_n$. As for standard polynomials, we can now define the notions of order and degree of a nonzero polynomial expression in $S[M]$: we call $\text{ord} \, f(x) := q_1$ and $\deg f(x) := q_n$ the \emph{order} and the \emph{degree} of $f(x)$, respectively. Finally, we call $c_n$ (resp., $c_n x^{q_n}$) the \emph{leading coefficient} (resp., \emph{leading term}) of the polynomial expression $f(x)$, while we call $c_1$ (resp., $c_1 x^{q_1}$) the \emph{order coefficient} (resp., \emph{order term}) of $f(x)$.

\bigskip
%%%%%%%%%%%%%%%%%%%%%%%%%%%%
%%%%%%%%%%%%%%%%%%%%%%%%%%%%
\section{Monoid Semidomains}
\label{sec:monoid semidomains}

The purpose of this section is twofold. On the one hand, we provide conditions under which the properties of almost atomicity and quasi-atomicity ascend to monoid semidomains. On the other hand, we enlarge the classes of known almost atomic and quasi-atomic semidomains.

\medskip
%%%%%%%%%%%%%%%%%%%%%%%%%%%%%%%%%%%%%%%%%%%%%%%%
\subsection{Ascent of Almost and Quasi-Atomicity} 

We start by establishing a sufficient condition for the ascent of almost atomicity and quasi-atomicity to monoid semidomains: $S[M]$ is almost atomic (resp., quasi-atomic) for any pair $(S,M)$ consisting of an almost atomic (resp., quasi-atomic) semidomain $S$ and a torsion-free reduced monoid~$M$ satisfying the ACCP.
\smallskip

Let $M$ be a torsion-free monoid, and let $S$ be a semidomain. We let $S^*x^{M}$ denote the monoid consisting of all the nonzero monomials of $S[M]$:
\[
    S^*x^{M} := \{sx^q : s \in S^* \text{ and } q \in M\}.
\]
Before proving the main theorem of this section, we need to argue the following lemma.

\begin{lemma}
    Let $M$ be a torsion-free monoid, and let $S$ be a semidomain. Then the following statements hold.
    \begin{enumerate}
        \item If $M$ is reduced, then $S^*$ is a divisor-closed submonoid of $S[M]^*$, and so a divisor-closed submonoid of $S^*x^{M}$.
        \smallskip
        
        \item $S^*x^{M}$ is a divisor-closed submonoid of $S[M]^*$.
    \end{enumerate}
\end{lemma}

\begin{proof}
    (1) It is clear that $S^*$ is a submonoid of both $S^* x^M$ and $S[M]^*$. Also, if $M$ is a reduced monoid, then it follows from~\eqref{eq:group of units of a semidomain} that $S[M]^\times = S^\times$ and so we immediately obtain that $S^*$ is a divisor-closed submonoid of both $S[M]^*$ and $S^*x^M$. 
%    \smallskip
%    from the fact that , one can readily obtain that  $S^*$ is a divisor-closed submonoid of $S^*x^M$. Thus, $S^*$ is a divisor-closed submonoid of $S[M]^*$.
    \smallskip
    
    (2) It follows immediately because every divisor of a nonzero monomial in $S[M]$ must be a nonzero monomial.
\end{proof}

It will be convenient at this point to generalize the notion of indecomposable polynomial given by Roitman in~\cite{mR93}. Let $R$ be a subring of $S[M]$, and let $f(x)$ be a nonzero polynomial expression in $R$. Then we say that $f(x)$ is \emph{indecomposable} in $R$ provided that it is not the product of two nonconstant polynomial expressions. We are in a position to establish the main result of this section.

\begin{theorem} \label{thm:ascent of AA/QA to monoid semidomains}
    Let $M$ be a reduced monoid satisfying the ACCP, and let $S$ be a semidomain such that the set of nonzero coefficients of each indecomposable polynomial expression in $S[M]$ has an MCD in $S$. Then the following statements hold.
    \begin{enumerate}
        \item If $S$ is almost atomic, then $S[M]$ is almost atomic.
        \smallskip

        \item If $S$ is quasi-atomic, then $S[M]$ is quasi-atomic.
    \end{enumerate}
\end{theorem}

\begin{proof}
    (1) Fix a nonzero polynomial expression $f(x)$ in $S[M]^*$. Now consider the set $L$ consisting of all indices $\ell \in \nn$ such that we can write $f(x) = g_1(x) \dots g_\ell(x)$ for some nonconstant polynomial expressions $g_1(x), \dots, g_\ell(x) \in S[M]$. The fact that $M$ satisfies the ACCP guarantees that the set $L$ is bounded. Therefore we can assume that $\ell$ is the largest index in~$L$. The maximality of $\ell$ immediately implies that the polynomial expressions $g_1(x), \dots, g_\ell(x)$ are indecomposable. As a consequence, for each $i \in \ldb 1,\ell \rdb$, we can guarantee the existence of a maximal common divisor $d_i \in S$ of the set of coefficients of the polynomial expression $g_i(x)$. Then we can write
	\begin{equation} \label{eq:factoring f}
		f(x) = d a_1(x) \cdots a_\ell(x),
	\end{equation}
	where $d := d_1 \cdots d_\ell$ and $a_i(x) := \frac{g_i(x)}{d_i} \in S[M]^*$ for every $i \in \ldb 1,\ell \rdb$. As the polynomial expressions $g_1(x), \dots, g_\ell(x)$ are indecomposable in $S[M]$, so are $a_1(x), \dots, a_\ell(x)$. This, along with the observation that, for any $i \in \ldb 1, \ell \rdb$, the only common divisors of the set of nonzero coefficients of $a_i(x)$ in $S^*$ are the units of $S$, ensures that $a_i(x)$ is an atom in $S[M]$. In addition, as $S^*$ is an almost atomic monoid, we can take an atomic element $s \in S^*$ such that $sd$ is also an atomic element of $S^*$. In fact, both $s$ and $sd$ are  atomic elements in $S[M]^*$ as we have seen earlier that $S^*$ is a divisor-closed submonoid of $S[M]^*$. This, together with the fact that $a_1(x), \dots, a_\ell(x)$ are atoms of $S[M]$, implies that $f(x)$ is an almost atomic element of $S[M]$: indeed, $sf(x) = (sd) a_1(x) \cdots a_\ell(x)$. Hence $S[M]$ is an almost atomic semidomain.
    \smallskip

    (2) Assume that $S$ is quasi-atomic, and fix a nonzero polynomial expression $f(x)$ in $S[M]$. Using the fact that $M$ satisfies the ACCP, we can proceed as we did in part (a) to write $f(x) = d a_1(x) \dots a_\ell(x)$, where $d$ is an element of $S$ and $a_1(x), \dots, a_\ell(x)$ are atoms of $S[M]$. Now the fact that $S$ is quasi-atomic allows us to pick $s \in S^*$ such that $sd$ is an atomic element of $S$, which in turn implies that $sd$ is an atomic element of $S[M]$ because $S^*$ is a divisor-closed submonoid of $S[M]^*$. Since $sd$ and $a_1(x), \dots, a_\ell(x)$ are all atomic elements in $S[M]$, the equality $sf(x) = (sd)a_1(x) \dots a_\ell(x)$ ensures that $f(x)$ is a quasi-atomic element of $S[M]$. Hence we conclude that $S[M]$ is a quasi-atomic semidomain.
\end{proof}

In the statement of Theorem~\ref{thm:ascent of AA/QA to monoid semidomains}, one of the hypotheses is that the exponent monoid $M$ satisfies the ACCP, which is a condition that may not seem sharp enough \emph{a priori}. As we will illustrate in the following example, such condition is the best one can hope: indeed, Section~\ref{sec:monoid domains} is devoted to the construction of a class $\{M_p : p \in \pp\}$ consisting of rank-one torsion-free reduced atomic monoids such that none of the monoid domains in the set $\{\ff_p[M_p] : p \in \pp \}$ is quasi-atomic. In the following example we provide a brief outline of the construction that motivated our work in Section~\ref{sec:monoid domains}.

% there are reduced atomic monoids $M$ whose monoid semidomain $S[M]$ quasi-atomic. we will take a first look at the technique that will allow us to show that neither almost atomicity nor quasi-atomicity ascend to monoid domains over finite fields.

\begin{example} \cite[Section~5]{CG19}
    Let $\zz\big[\frac12\big]_{\ge 0} \subseteq M \subseteq \zz\big[\frac12, \frac13\big]_{\ge 0}$ be an extension of Puiseux monoids, which are rank-one torsion-free monoids, and assume that the intermediate monoid $M$ is atomic. A monoid satisfying these conditions was explicitly constructed in \cite[Proposition~5.1]{CG19}. Observe that $M$ does not satisfy the ACCP because $\big(\frac1{2^n} + M \big)_{n \ge 1}$ is an ascending chain of principal ideals of $M$ that does not stabilize. Thus, $M$ is a reduced atomic monoid that does not satisfy the ACCP. On the other hand, as $\ff_2$ is a field, it trivially follows that the set of nonzero coefficients of each indecomposable polynomial in $\ff_2[M]$ has an MCD in $\ff_2$. Finally, as we shall see in the proof of Theorem~\ref{thm:non-ascent of QA in characteristic p}, although $\ff_2$ is trivially almost atomic (and so quasi-atomic), the monoid domain $\ff_2[M]$ is not even quasi-atomic.
\end{example}

\medskip
%%%%%%%%%%%%%%%%%%%%%%%%%%%%%%%%%%%%%%%%%%%%%%%%
\subsection{Another Construction of Almost and Quasi-Atomic Domains}

The sufficient condition for the ascent of atomicity given in the previous section allows us to enlarge the class of known almost atomic and quasi-atomic semidomains. We proceed to identify further classes of almost atomic and quasi-atomic domains inside polynomial rings. Our new setting consists of a tower of integral domains $S \subseteq F \subseteq K$, where $F$ is the quotient field of $R$ and $K$ is a field extension of $F$. Our target integral domains are the subrings
\begin{equation} \label{eq:subring of polynomias}
    S[x] + F[x]x^2 \quad \text{and} \quad S[x] + K[x]x^2
\end{equation}
of the polynomial domain $K[x]$. Observe that $S^*$ is a divisor-closed submonoid of the multiplicative monoids of both $S[x] + F[x]x^2$ and $S[x] + K[x]x^2$. Before providing sufficient conditions for the integral domains in~\eqref{eq:subring of polynomias} to be almost atomic and quasi-atomic, respectively, it is convenient to argue the following lemma, which describes the multiplicative submonoid consisting of all atomic elements of $S[x] + K[x]x^2$.

\begin{lemma} \label{lem:atomic elements of S[x]+F[x]x^2}
    Let $K$ be a field, and let $S$ be a subring of $K$ that satisfies the ACCP. For a nonzero polynomial $f(x) \in S[x] + K[x]x^2$, the following statements hold.
    \begin{enumerate}
        \item $f(x)$ is atomic if its order coefficient belongs to $S$.
        \smallskip

        \item If $S$ is not a field, then $f(x)$ is atomic if and only if its order coefficient belongs to $S$.
    \end{enumerate}
\end{lemma}

\begin{proof}
    Set $R := S[x] + K[x]x^2$. Since $S^*$ is a divisor-closed submonoid of $R^*$, the units of $R$ are precisely the units of $S$, that is, $R^\times = S^\times$.
    \smallskip
    
    (1) Assume that the order coefficient of $f(x)$ belongs to $S$, and let us show that $f(x)$ is atomic. Towards this, we first observe that, since the set of divisors of $x$ in $R$ is $S^\times \cup S^\times x$, the monomial $x$ must be an atom. This, along with the fact that $S$ is atomic (because $S$ satisfies the ACCP), ensures that every monomial of the form $sx^n$, where $s \in S^*$ and $n \in \nn_0$ is an atomic element of $R$. Thus, if $\text{ord} \, f(x) \ge 2$, then we can replace $f(x)$ by $f(x)/x^{\max\{0, -1+\text{ord} \, f(x) \} }$ to guarantee $\text{ord} \, f(x) \in \{0,1\}$, which we assume for the rest of this paragraph. It is clear that the set consisting of all $\ell \in \nn$ such that $f(x)$ can be written as the product of $\ell$ nonconstant polynomials in $R$ is bounded. Hence we can write $f(x) = g_1(x) \dots g_\ell(x)$ for some nonconstant polynomials $g_1(x), \dots, g_\ell(x) \in R$, assuming that $\ell$ was taken as large as it can possibly be. This immediately implies that $g_i(x)$ is indecomposable for every $i \in \ldb 1, \ell \rdb$. In addition, for each $i \in \ldb 1,\ell \rdb$, the fact that $\text{ord} \, g_i(x) \le \text{ord} \, f(x) \le 1$ ensures that the order coefficient of $g_i(x)$ belongs to $S$, and so the fact that $S$ satisfies the ACCP guarantees that every ascending chain of principal ideals of $S$ starting at the order coefficient of $g_i(x)$ must stabilize. Hence, for each $i \in \ldb 1, \ell \rdb$, we can take $d_i \in S^*$ such that $\frac{g_i(x)}{d_i}$ is an atom in $R$. As $f(x) = (d_1 \cdots d_\ell) \prod_{i=1}^\ell \frac{g_i(x)}{d_i}$, the fact that $S$ is atomic implies that $f(x)$ is an atomic element.
    \smallskip

    (2) Now assume that $S$ is not a field, which implies that $S^* \setminus S^\times$ is a nonempty set. We have seen in the previous part that if the order coefficient of $f(x)$ belongs to $S$, then $f(x)$ is atomic. Conversely, assume that $f(x)$ is atomic. If $\text{ord} \, f(x) \in \{0,1\}$, then the order coefficient of $f(x)$ belongs to $S$ by the definition of~$R$. Assume therefore that $\text{ord} \, f(x) \ge 2$. Write $f(x) = a_1(x) \cdots a_\ell(x)$, where $a_1(x), \dots, a_\ell(x)$ are atoms in $R$. Observe that $\text{ord} \, a_i(x) \in \{0,1\}$ for every $i \in \ldb 1, \ell \rdb$: indeed, if $\text{ord} \, a_i(x) \ge 2$ for some $i \in \ldb 1,\ell \rdb$, then one could take any element $s \in S^* \setminus S^\times$ to non-trivially factor $a_i(x)$ in $R$ as $s \frac{a_i(x)}s$, contradicting that $a_i(x)$ is an atom in $R$. As the order of every factor on the right-hand side of $f(x) = a_1(x) \cdots a_\ell(x)$ belongs to $\{0,1\}$, their order coefficients must belong to $S$, which implies that the order coefficient of $f(x)$ belongs to $S$.
\end{proof}

We proceed to characterize when the integral domain $S[x] + F[x]x^2$ and $S[x] + K[x]x^2$ are almost atomic and quasi-atomic, respectively. We would like to emphasizing that the first prototypical example of the form~\eqref{eq:subring of polynomias} were provided by Lebowitz-Lockard in~\cite{nL19}.

\begin{prop} \label{prop:classes of AA/QA domains}
    Let $S$ be an integral domain that satisfies the ACCP whose quotient field is $F$, and let $F \subseteq K$ be an extension field. Then the following statements hold.
    \begin{enumerate}
        \item $F[x] + K[x]x^2$ satisfies the ACCP and so is atomic.
        \smallskip
        
        \item $S[x] + F[x]x^2$ is almost atomic.
        \smallskip

        \item $S[x] + K[x]x^2$ is quasi-atomic. %atomic if $F$ is not the quotient field of $S$.
    \end{enumerate}
\end{prop}

\begin{proof}
    (1) Set $R := F[x] + K[x]x^2$. Clearly, $R^\times = F^\times$, and so $K[x]^\times \cap R = K^\times \cap R = F^\times = R^\times$. This, along with the fact that $K[x]$ satisfies the ACCP (indeed, it is a UFD), implies that $R$ satisfies the ACCP \cite[Proposition~2.1]{aG74}.
    \smallskip
    
    (2) Set $R := S[x] + F[x]x^2$, and let us show that $R$ is almost atomic. Fix a nonzero polynomial $f(x)$ in $R$, and let $q$ be the order coefficient of $f(x)$. Since $q$ is a nonzero element of $F$, which is the quotient field of $S$, we can take a nonzero $s \in S$ such that $sq \in S$. Since $S$ satisfies the ACCP, the element $s$ is atomic in $S$, and so it is also atomic in $R$ because $S^*$ is a divisor-closed submonoid of $R^*$. On the other hand, $sf(x)$ is a polynomial in $R$ whose order coefficient belongs to $S$, and so $sf(x)$ is atomic in $R$ by virtue of part~(1) of Lemma~\ref{lem:atomic elements of S[x]+F[x]x^2}. Hence $R$ is almost atomic.
    \smallskip
    
    (3) Now set $R := S[x] + K[x]x^2$, and let us show that $R$ is quasi-atomic. Let $f(x)$ be a nonzero polynomial in $R$, and let $k \in K$ denote the order coefficient of $f(x)$. Since $K$ is a field, we know that the monomial $k^{-1}x^2$ belongs to $R$. Set $g(x) := (k^{-1}x^2) f(x)$. Then $g(x)$ is a polynomial in $K[x]$ whose order is at least~$2$, whence $g(x) \in R$. In addition, the order coefficient of $g(x)$ is $1$. Thus, it follows from part~(1) of Lemma~\ref{lem:atomic elements of S[x]+F[x]x^2} that $g(x)$ is atomic in $R$, and so $f(x)$ is a quasi-atomic element in $R$. Hence we conclude that $R$ is a quasi-atomic domain. 
\end{proof}

We can now use Proposition~\ref{prop:classes of AA/QA domains} to produce examples of almost atomic (resp., quasi-atomic) domains that are not atomic (resp., almost atomic).

\begin{example} \label{ex:AA/QA not QA/AA}
	Let $S$ be an integral domain that satisfies the ACCP but is not a field, and let $F$ be the quotient field of $S$.
    \begin{enumerate}
        \item In light of part~(2) of Proposition~\ref{prop:classes of AA/QA domains}, the integral domain $R := S[x] + F[x]x^2$ is almost atomic. However, $R$ is not atomic because no monomial of degree $2$ in $R$ is an atomic element: indeed, after fixing any $q \in F^\times$, one can take $a \in \mathcal{A}(S)$ (such an element $a$ must exits because $S$ is atomic but not a field) and non-trivially factor $qx^2$ as $a \big(\frac{q}a x^2 \big)$. Thus, $R$ is an almost atomic domain that is not atomic. The special case $\zz[x] + \qq[x]x^2$ was first discussed in~\cite[Section~2]{nL19}.
        \smallskip

        \item Let $K$ be an extension field of $F$ such that $F \subsetneq K$, and now set $R := S[x] + K[x]x^2$. It follows from part~(3) of Proposition~\ref{prop:classes of AA/QA domains} that $R$ is a quasi-atomic domain. Now fix $\kappa \in K \setminus F$, and consider the monomial $\kappa x^2$, which is an element of $R$. In light of part~(1) of Lemma~\ref{lem:atomic elements of S[x]+F[x]x^2}, the order coefficient of each atomic element of $R$ belongs to $S$, whence $f(x)(\kappa x^2)$ cannot be atomic in $R$ for any atomic element $f(x)$ in $R$. Thus, $R$ is a quasi-atomic domain that is not almost atomic. The special case $\zz[x] + \rr[x]x^2$ was first discussed in~\cite[Section~2]{nL19}.
    \end{enumerate}
\end{example}

\bigskip
%%%%%%%%%%%%%%%%%%%%%%%%%%%%%%%
%%%%%%%%%%%%%%%%%%%%%%%%%%%%%%%
\section{Polynomial Extensions}
\label{sec:polynomial rings}

As mentioned in the introduction, the property of being atomic does not ascend to polynomial extensions, and this is due to Roitman~\cite{mR93}. Motivated by this result, we argue in this section a similar result for quasi-atomicity. In order to prove that quasi-atomicity does not ascend to polynomial extensions, we will prove, in the next theorem, that the polynomial domain of
\begin{equation} \label{eq:QA domain not AA}
	R := \mathbb{Z} + \mathbb{Z}x + x^2\mathbb{R}[x]
\end{equation}
is not quasi-atomic (we have already seen in Example~\ref{ex:AA/QA not QA/AA} that $R$ is a quasi-atomic domain). It suffices to argue that the polynomial extension of $R$ is not quasi-atomic, and we will do so in the next theorem.
\smallskip

Before proving the main theorem of this section, we introduce some notation that will be helpful. Let $R$ be a commutative ring with identity. For a nonzero polynomial $f \in R[x,y]$, we can write $f_y(x) = \sum_{i=0}^m b_i(y) x^i \in R[x]$ (where $b_0(y), \dots, b_m(y) \in \rr[y]$) or $f_x(y) = \sum_{i=0}^n c_i(x) y^i \in R[y]$ (where $c_0(x), \dots, c_n(x) \in \rr[x]$): we call $\deg f_y$ and $\deg f_x$ the $x$-\emph{degree} and the $y$-\emph{degree} of $f$, respectively.

\begin{theorem}
	Let $R$ be the integral domain in~\eqref{eq:QA domain not AA}, and let $y$ be an (algebraically) independent variable over $R$. Then the polynomial extension $R[y]$ is not quasi-atomic.
\end{theorem}

\begin{proof}
	Let $\alpha$ and $\beta$ be two (distinct) positive irrational numbers such that $\alpha, \beta$ are algebraically independent over $\qq$. Now consider the following two elements of $R$:
	\[
		a := \alpha x^2 \quad \text{ and } \quad b := \beta x^2.
	\]
	We are done once we prove that the linear polynomial $ay + b$ is not quasi-atomic in $R[y]$. Assume, towards a contradiction, that $ay + b$ is quasi-atomic. Thus, there exists a nonzero polynomial $F(y) \in R[y]$ such that $F(y)(ay + b)$ factors into irreducibles in $R[y]$. Write
	\begin{equation}
		F(y)(ay + b) = \prod_{i=1}^n A_i(y),
	\end{equation}
	where $A_1(y), \dots, A_n(y)$ are irreducibles in $R[y]$. As $0 \notin \{a,b\}$, after simplifying powers of $y$ on both sides of the previous equality, we can assume that $A_i(y)$ is not associate with $y$ for any $i \in \ldb 1,n \rdb$. Since $\rr[x,y]$ is an atomic domain (indeed, $\rr[x,y]$ is a UFD), for each $i \in \ldb 1,n \rdb$, there exists $\ell_i \in \nn$ such that
	\begin{equation} \label{eq:quasi-atomic product refined in R[x,y]}
		F(y)(ay + b) = \prod_{i=1}^n \prod_{j=1}^{\ell_i} a_{i,j}(x,y), %%% A_i := G_i
	\end{equation}
%	\[
%		A_i(y) = \prod_{j=1}^{\ell_i} a_{i,j}(x,y) %%%  a_{i,j} := H_{i,j}
%	\]
	for irreducible polynomials $a_{i,1}(x,y), \dots, a_{i,\ell_i}(x,y)$ in $\rr[x,y]$ such that $A_i(y) = a_{i,1}(x,y) \cdots a_{i,\ell_i}(x,y)$ for every $i \in \ldb 1,n \rdb$. Since~\eqref{eq:quasi-atomic product refined in R[x,y]} takes place in $\rr[x,y]$, which is a UFD, we can write
	\begin{equation} \label{eq:factorization of ay+b}
		ay+b = r \prod_{(i,j) \in S} a_{i,j}(x,y),
	\end{equation}
	for some nonzero $r \in \rr$ and a (nonempty) subset $S$ of $\ldb 1,n \rdb \times \ldb 1, \max_{i} \ell_i \rdb$. Now as $ay + b$ has $y$-degree~$1$, there exists exactly one irreducible factor $a(x,y) := a_{(i_0, j_0)}(x,y)$ on the right-hand side of~\eqref{eq:factorization of ay+b} %in $\{a_{i,j}(x,y) : (i,j) \in \ldb 1,n \rdb \times \ldb 1, \ell_i \rdb \}$ 
	having $y$-degree $1$ while the rest of such factors have $y$-degree $0$, that is, belong to $\rr[x]$. Thus, we can write $ay + b = a(x,y)B(x)$ for some nonzero polynomial $B(x) \in \rr[x]$. As $a$ and $b$ are both monomials in $\rr[x]$, so is $B(x)$. Write $B(x) = \gamma x^e$ for a nonzero coefficient $\gamma \in \rr$ and an exponent $e \in \ldb 0, 2 \rdb$. 
	Thus, %Then we can rewrite~\eqref{eq:factorization of ay+b} as follows:
	\[
		a(x,y) = \frac{\alpha}{\gamma}x^{2-e}y + \frac{\beta}{\gamma} x^{2-e}.
	\]
	Because $a(x,y)$ is an irreducible factor of $A_{i_0}(y)$ in $\rr[x,y]$, we can write $A_{i_0}(y) = a(x,y) t(x,y)$ for some $t(x,y) \in \rr[x,y]$. Let $d$ be the $y$-degree of $t(x,y)$, and take $T_0(x), \dots, T_d(x) \in \rr[x]$ such that $t(x,y) = \sum_{i=0}^{d} T_i(x)y^i$. Therefore
	\begin{align} \label{eq:A_{i_0} and T_i's}
		A_{i_0}(y) = \bigg( \frac{\alpha}{\gamma} x^{2-e} y + \frac{\beta}{\gamma} x^{2-e} \bigg) \sum_{i=0}^d T_i y^i
						  =  \frac{\alpha}{\gamma} x^{2-e} T_d y^{d+1} + \bigg( \sum_{i=1}^d x^{2-e} \Big( \frac{\beta}{\gamma} T_i + \frac{\alpha}{\gamma} T_{i-1} \Big) y^i \bigg) + \frac{\beta}{\gamma} x^{2-e} T_0.
	\end{align}
	For each index $i \in \ldb 0,d+1\rdb$, we can set $p_i(x) := \frac{\beta}{\gamma}  x^{2-e} T_i(x) + \frac{\alpha}{\gamma}  x^{2-e} T_{i-1}(x)$. After considering $T_{-1}(x)=T_{d+1}(x)=0$, we can write $A_{i_0}(y)$ as follows:
\begin{equation} \label{eq:A_{i_0}}
	A_{i_0}(y) = \sum_{j=0}^{d+1} p_j(x)y^j.
\end{equation}
Observe that $\text{ord} \, T_j(x) < e$ for some $j \in \ldb 0,d \rdb$ as otherwise $2$ and $ \sum_{j=0}^{d+1} \frac{p_j(x)}2 y^j$ would be two nonunits in $R[y]$ and so $A_{i_0}(y) = 2 \big( \sum_{j=0}^{d+1} \frac{p_j(x)}2 y^j \big)$ would be a nontrivial factorization of $A_{i_0}(y)$ in $R[y]$. However, this is not possible because $A_{i_0}(y)$ is irreducible. As a consequence, the set
\[
	D := \{i \in \ldb 0,d  \rdb : \text{ord} \, T_i(x) < e \}
\]
is nonempty. The set of indices $D$ will play a crucial role throughout the rest of this proof.

For each $i \in \ldb 0,d \rdb$, we let $c_i$ denote the order coefficient of $T_i(x)$, letting $c_i=0$ when $T_i(x)$ is the zero polynomial. Because $A_{i_0}(y)$ is an irreducible polynomial in $R[y]$ that is not associate with $y$, its constant coefficient, $\frac{\beta}{\gamma}x^{2-e}T_0(x)$ is nonzero, whence $T_0(x) \neq 0$ and $c_0 \neq 0$. Let us prove the following claim.
	\smallskip

	\noindent \textsc{Claim 1.} For each $i \in \ldb 0, d \rdb$ such that $T_i(x) \neq 0$ and $\text{ord} \, T_i(x) < e$, the following statements hold.
	\begin{itemize}
		\item If $i = 0$, then $\frac{\beta}{\gamma} c_0 \in \zz$. %the order coefficient of $p_0(x)$ is $\frac{\beta}{\gamma} c_0$, and so $\frac{\beta}{\gamma} c_0 \in \zz$.
		\smallskip
		
		\item If $i \in \ldb 1,d \rdb$ and $T_{i-1} \neq 0$,
		%(\textcolor{red}{$T_{i-1}(x) = 0$?}), 
%		5
%		5
%		5
%		5
%		5
%		5
		then $\frac{\beta}{\gamma} c_i \in \zz$ when $\text{ord} \, T_i(x) < \text{ord} \, T_{i-1}(x)$, and $\frac{\alpha}{\gamma} c_{i-1} \in \zz$ when $\text{ord} \, T_i(x) > \text{ord} \, T_{i-1}(x)$.
		\smallskip
		
		\item If $i=d$, then $\frac{\alpha}{\gamma} c_d \in \zz$.
	\end{itemize}
	\smallskip
	
	\noindent \textsc{Proof of Claim 1.} Fix $i \in \ldb 0,d \rdb$ such that $\text{ord} \, T_i(x) < e$. First suppose that $i=0$. The polynomial $\frac{\beta}{\gamma}x^{2-e} T_0(x) \in \rr[x]$ belongs to $R$ because it is the constant coefficient of $A_{i_0}(y) \in R[y]$. Also, from the assumption $\text{ord} \, T_0(x) < e$, we obtain that $\text{ord} \, \frac{\beta}{\gamma} x^{2-e}T_0(x) < 2$, which implies that the order coefficient of $\frac{\beta}{\gamma} x^{2-e}T_0(x)$ must be an integer: $\frac{\beta}{\gamma} c_0 \in \zz$. The other cases are similar. Therefore Claim 1 is established.
	\smallskip
	
%	\textcolor{red}{Maybe we want to analyze first $D = \{0\}$.}
%	\bigskip
%%
%
%5
%5
%5
%5
%5
%5
%5
%5
%5
%5
%
%4
	
	Set $s := \min D$, and let us verify the following two inequalities from Claim~1.
	\begin{itemize}
		\item  $d \ge 1$. It suffices to note that if the equality $d=0$ held, then $\text{ord} \, T_0(x) = \text{ord} \, T_d(x)< e$ (we have already seen that $T_0(x) \neq 0$), and so Claim~1 would imply that $\alpha,\beta \in \frac{\gamma}{c_0} \zz$, which is not possible because $\alpha$ and $\beta$ are linearly independent over~$\qq$. 
		\smallskip
		
		\item $s<d$. Suppose otherwise that $s=d$. Therefore it follows from the minimality of $s$ that $\text{ord} \, T_{d-1}(x) \ge e > \text{ord} \, T_d(x)$, whence $\frac{\beta}{\gamma} c_d \in \zz$ by Claim~1. Another application of Claim~1 guarantees that $\frac{\alpha}{\gamma} c_d \in \zz$. As a result, $\alpha,\beta \in \frac{\gamma}{c_d}\zz$. Once again this is not possible because $\alpha$ and $\beta$ are linearly independent over $\qq$.
	\end{itemize}
	\medskip
	
	Let $o_s$ denote the order of the polynomial $T_s(x)$. Our next goal is to establish the following claim, after which we will be in a position to finish our proof.
	\medskip
	
	\noindent \textsc{Claim 2.} For each $k \in \ldb s+1,d \rdb$, there exists a nonconstant polynomial $Q_k(x) \in \qq[x]$ such that $Q_k\big(\frac{\alpha}{\beta}\big)c_s$ is the coefficient of degree $o_s$ of $T_k(x)$.
	\smallskip
	
	\noindent \textsc{Proof of Claim 2.} We proceed by induction on $k$. For the base case, assume that $k=s+1$ and consider the two cases.
	
	First, assume that $s=0$. This means that  the inequality $\text{ord} \, T_0(x) < e$ holds, and so it follows from Claim~1 that $\frac{\beta}{\gamma} c_0 \in \zz$. Therefore we can write $\gamma$ as follows:
	\begin{equation} \label{eq:gamma}
		\gamma = \frac{c_0}{z_0} \beta
	\end{equation}
	for some nonzero $z_0 \in \zz$ (both $\beta$ and $c_0$ are nonzero). We know that $p_1(x) = \frac{\beta}{\gamma} x^{2-e} T_1(x) + \frac{\alpha}{\gamma} x^{2-e} T_0(x) \in R$. Note $T_1(x) \neq 0$ as otherwise $\frac{\alpha}{\gamma}x^{2-e} T_0(x) = p_1(x) \in R$ and so $\text{ord} \, \frac{\alpha}{\gamma}x^{2-e} T_0(x) < 2$ would imply that $\frac{\alpha}{\gamma} c_0 = z_1$ for some $z_1 \in \zz$;
%	%
%	5
%	5
%	5
%	5
%	5
%	55
%	5
%	(\textcolor{red}{this should be covered in Claim 1}); 
	however, in this case, $\frac{c_0}{z_1} \alpha = \frac{c_0}{z_0} \beta$ by virtue of~\eqref{eq:gamma}, which is not possible because $\alpha$ and $\beta$ are algebraically independent over~$\qq$. In a similar way, we can deduce from Claim~1 that $\text{ord} \, T_1(x) \le \text{ord} \, T_0(x)$. Let $c_{1,o_0}$ denote the coefficient of degree $o_0$ of $T_1(x)$. Since $p_1(x)= \frac{\beta}{\gamma} x^{2-e} T_1(x) + \frac{\alpha}{\gamma} x^{2-e} T_{0}(x) \in R$, we know that the coefficient of degree $\text{ord} \, \frac{\alpha}{\gamma} x^{2-e} T_0(x)$ in $p_1(x)$ must be an integer, which implies that $\frac{\beta}{\gamma}c_{1,o_0}+ \frac{\alpha}{\gamma}c_0=z_1$ for some $z_1 \in \mathbb{Z}$. After rearranging this last equation and substituting (\ref{eq:gamma}), we obtain $c_{1,o_0}=\frac{z_1}{z_0}c_0-\frac{\alpha}{\beta}c_0$. Then, we are done if we take $Q_1(x)=-x+\frac{z_1}{z_0}$.
	Now assume that $s \ge 1$, and write
	\begin{equation} \label{eq:HH}
		p_s(x) = \frac{\beta}{\gamma} x^{2-e} T_s(x) + \frac{\alpha}{\gamma} x^{2-e} T_{s-1}(x) \in R.
	\end{equation}
	 It follows from the minimality of $s$ that $\text{ord} \, T_{s-1}(x) \ge e > \text{ord} \, T_s(x)$, and so Claim~1 guarantees that $\frac{\beta}{\gamma} c_s = z_s$ for some nonzero $z_s \in \zz$. Since $p_{s+1}(x) = \frac{\beta}{\gamma} x^{2-e} T_{s+1}(x) + \frac{\alpha}{\gamma} x^{2-e} T_s(x) \in R$ and $\text{ord} \,  \frac{\alpha}{\gamma} x^{2-e} T_s(x)  < 2$, the coefficient of $p_{s+1}(x)$ of degree $\text{ord} \, \frac{\alpha}{\gamma} x^{2-e} T_s(x)$ belongs to $\zz$, which means that
	\begin{equation} \label{eq:one more coefficient}
		\frac{\beta}{\gamma} c_{s+1,o_s} + \frac{\alpha}{\gamma}c_s = z_{s+1}
	\end{equation}
	for some $z_{s+1} \in \zz$, where $c_{s+1, o_s}$ is the coefficient of degree $o_s$ of the polynomial $T_{s+1}(x)$. After substituting $\frac{\beta}{\gamma} = \frac{z_s}{c_s}$ in~\eqref{eq:one more coefficient} and taking $Q_{s+1}(x) = -x + \frac{z_{s+1}}{z_s} \in \qq[x]$, we see that
	\[
		c_{s+1,o_s} = \Big(\frac{z_{s+1}}{z_s} -  \frac{\alpha}{\beta}\Big)c_s = Q_{s+1}\Big(\frac{\alpha}{\beta}\Big) c_s, %\frac{z_s}{c_s} c_{s+1,o_s} + \frac{z_s \alpha}{\beta c_s}c_s = z
	\]
	which completes our base case.
	
	To argue our inductive step, we assume that for some $k \in \ldb s+2,d-1 \rdb$ there exists a nonconstant polynomial $Q_{k-1}(x) \in \qq[x]$ such that the $Q_{k-1}\big(\frac{\alpha}{\beta}\big)c_s$ is the coefficient of degree $o_s$ of $T_{k-1}(x)$. As $o_s < e$, %= \text{ord} \, x^{2-e}T_s(x) < 2$, T
	the coefficient of degree $2-e+o_s$ of $p_k(x) = \frac{\beta}{\gamma} x^{2-e} T_{k}(x) + \frac{\alpha}{\gamma} x^{2-e} T_{k-1}(x)$ belongs to $\zz$, which means that
	\[
		c_{k, o_s} + \frac{\alpha}{\beta} Q_{k-1}\Big( \frac{\alpha}{\beta}\Big)c_s = Z\frac{\gamma}{\beta}
	\]
	for some $Z \in \zz$. It follows now from the minimality of $s$ that $\text{ord} \, T_{s-1}(x) \ge e > \text{ord} \, T_s(x)$, and so it follows from Claim~1 that $\frac{\beta}{\gamma} c_s = z_s$ for some $z_s \in \zz$. Therefore
	\[
		c_{k, o_s}= \frac{Z}{z_s} c_s - \frac{\alpha}{\beta} Q_{k-1}\Big( \frac{\alpha}{\beta}\Big)c_s = Q_k\Big( \frac{\alpha}{\beta}\Big) c_s
	\]
	for the nonconstant polynomial $Q_k(x) = \frac{Z}{z_s} - xQ_{k-1}(x) \in \qq[x]$. This completes the proof of the inductive step, and so Claim 2 has been established.
	\smallskip
	
	To conclude our proof, let $\ell$ be the maximum index in $\ldb 0,d \rdb$ such that $\text{ord} \, T_\ell(x) < e$, and set $o := \text{ord} \, x^{2-e} T_s(x) < 2$. First, we assume that $\ell < d$.
%	%\color{red}{If $\ell = d$, then TODO...} 
%	%
%	%%%%%
%	%
%	5
%	5
%	5
%	5
%	55
%	%
	Then we can say that $\text{ord} \, x^{2-e}T_\ell(x) < 2$ while $\text{ord} \, x^{2-e} T_j(x) \ge 2$ for every $j \in \ldb \ell+1, d \rdb$. Let  $c_{\ell,o}$ be the degree-$o$ coefficient of the polynomial  $\frac{\alpha}{\gamma} x^{2-e} T_\ell(x)$. Since $\text{ord} \, x^{2-e}T_{\ell + 1} \ge 2$ and $p_{\ell+1}(x) = \frac{\alpha}{\gamma} x^{2-e} T_\ell(x) + \frac{\beta}{\gamma} x^{2-e} T_{\ell+1}(x) \in R$, then $c_{\ell,o} \in \zz$. Also, notice that from $\text{ord} \, T_{s-1} \ge e > \text{ord} \, T_s(x)$, we obtain that $\frac1{\gamma} = \frac{z_s}{c_s} \frac1{\beta}$ for some $z_s \in \zz$. Now we can use Claim~2 to pick  a nonconstant polynomial $Q_\ell(x) \in \qq[x]$ such that $Q_\ell\big(\frac{\alpha}{\beta} \big)c_s$ is the degree-$o$ coefficient of $x^{2-e}T_\ell(x)$, which implies that
	\[
		%\frac{\alpha}{\gamma} x^{2-e} T_\ell(x) = \frac{z_s}{c_s}  \frac{\alpha}{\beta} x^{2-e} T_\ell(x) \quad \text{and so}
			z_s  \frac{\alpha}{\beta} Q_\ell\Big(\frac{\alpha}{\beta}\Big)  = \frac{z_s}{c_s}  \frac{\alpha}{\beta} Q_\ell\Big(\frac{\alpha}{\beta}\Big)c_s = \frac{\alpha}{\gamma}Q_\ell\Big(\frac{\alpha}{\beta} \Big)c_s = c_{\ell,o}.
	\]
	Therefore $Q(x) := z_s xQ_\ell(x) - c_{\ell,o} \in \qq[x]$ is a nonzero polynomial having $\frac{\alpha}{\beta}$ as a root. However, this contradicts that $\alpha$ and $\beta$ were initially taken algebraically independent. The case $\ell = d$ follows similarly. Therefore we conclude that $R[y]$ is not quasi-atomic.
\end{proof}

We conclude this section with the following related question.

\begin{question}
	Does almost atomicity ascend from any integral domain to its polynomial extension?
\end{question}

\bigskip
%%%%%%%%%%%%%%%%%%%%%%%
%%%%%%%%%%%%%%%%%%%%%%%
\section{Monoid Domains over Finite Fields}
\label{sec:monoid domains}

Our primary goal in this section is to argue that almost atomicity and quasi-atomicity do not ascend to monoid domains over fields. More specifically, we will prove that for any $p \in \pp$, there exists a rank-one torsion-free atomic monoid $M$ such that the monoid domain $\ff_p[M]$ is not quasi-atomic. The strategy we will follow is the one introduced in \cite[Section~5]{CG19}, where a similar result was proved for atomicity in the case when $p=2$.
\smallskip

\begin{theorem}
	For each $p \in \pp$, there exists a rank-one torsion-free atomic monoid $M_p$ such that the monoid algebra $\ff_p[M_p]$ is not even quasi-atomic.
\end{theorem}

\medskip
%%%%%%%%%%%%%%%%%%%%%%%%%%%%%
\subsection{Construction of the Monoid of Exponents}

Let us proceed to construct, for each pair $(q,r) \in \pp \times \zz_{\ge 2}$ with $\gcd(q,r) = 1$, an atomic Puiseux monoid that we will use later as exponent monoids of certain monoid algebras.

\begin{prop} \label{prop:atomic PMs}
	For every pair $(q,r) \in \pp \times \zz_{\ge 2}$, there exists a Puiseux monoid $M_{q,r}$ satisfying the following properties:
	\begin{enumerate}
		\item $\nn_0\big[\frac1q\big] \subseteq M_{q,r} \subseteq \zz\big[\frac1q, \frac1r\big]$, and
		\smallskip
		
		\item $M_{q,r}$ is atomic.
	\end{enumerate}
\end{prop}

\begin{proof}
	Let $(\ell_n)_{n \ge 1}$ be a strictly increasing sequence of positive integers such that $r^{\ell_n - \ell_{n-1}} > 2q^{n+1}$ for every $n \in \nn$, and let $L$ denote the underlying set of $(\ell_n)_{n \ge 1}$. For each $n \in \nn$, set
	\[
		a_n := \frac{q^n r^{\ell_n} - 1}{2 q^{2n} r^{\ell_n}} \quad \text{ and } \quad  b_n := \frac{q^n r^{\ell_n} + 1}{2 q^{2n} r^{\ell_n}}.
	\]
	Clearly, $1 > b_n > a_n$ for every $n \in \nn$. In addition, after replacing in the expression $\frac1{2 q^{2n}r^{\ell_n}} + \frac1{2q^{2n+2}r^{\ell_{n+1}}}$ both $q^{2n}$ and $q^{2n+2}$ by $q^{n+1}$ and both $r^{\ell_n}$ and $r^{\ell_{n+1}}$ by $2$, we see that $\frac{q-1}{2q^{n+1}} > \frac1{2 q^{2n}r^{\ell_n}} + \frac1{2q^{2n+2}r^{\ell_{n+1}}}$, from which one obtains the following:
	\[
		a_n = \frac1{2q^n} - \frac1{2q^{2n}r^{\ell_n}} = \frac1{2q^{n+1}} + \bigg( \frac{q-1}{2q^{n+1}} - \frac1{2q^{2n}r^{\ell_n}} \bigg) > \frac1{2q^{n+1}} + \frac1{2q^{2n+2}r^{\ell_{n+1}}} = b_{n+1}.
	\]
	Thus, $a_n > b_{n+1}$ for every $n \in \nn$, and so the sequence $b_1, a_1, b_2, a_2, \ldots$ is a strictly decreasing sequence of rationals in $(0,1)$. Now let $M := M_{q,r}$ denote the Puiseux monoid generated by the set
	\[
		A := \{a_n, b_n : n \in \nn\}.
	\]
	As $\frac1{q^n} = a_{n} + b_{n} \in M$ for every $n \in \nn$, the inclusion $\nn_0\big[\frac1q\big] \subseteq M_{q,r}$ holds. On the other hand, the fact that $A$ is a subset of $ \zz\big[\frac1q, \frac1r\big]$ ensures that $M_{q,r} \subseteq \zz\big[\frac1q, \frac1r\big]$. Hence $M$ satisfies property~(1). 
	\smallskip
	
	Therefore we are done once we argue that $M$ is atomic, which amounts to proving that $\mathcal{A}(M) = A$. Since $M$ is reduced, we only need to check that $A$ is a minimal generating set of $M$. We suppose, by way of contradiction, that $A$ is not a minimal generating set of~$M$, and split the rest of the proof into the following cases.
	\smallskip
	
	\textsc{Case 1:} $M = \langle A \setminus \{a_n\} \rangle$ for some $n \in \nn$. Since $a_n = \min \{a_i, b_i : i \in \ldb 1,n \rdb\}$, we see that $b_n \nmid_M a_n$ and $a_i, b_i \nmid_M a_n$ for any $i \in \ldb 1,n-1 \rdb$. This, along with the fact that $A \setminus \{a_n\}$ is a generating set of $M$, allows us to pick an index $N \in \nn$ with $N > n$ and  such that
	\begin{equation} \label{eq:expression for a_n for case 1}
		a_n = \sum_{i=n+1}^N \alpha_i a_i  + \sum_{i=n+1}^N \beta_i b_i
	\end{equation}
	for some coefficients $\alpha_{n+1}, \dots, \alpha_N \in \nn_0$ and $\beta_{n+1}, \dots, \beta_N \in \nn_0$. We can further assume that $\alpha_N \neq 0$ or $\beta_N \neq 0$. Observe that $\alpha_i \neq \beta_i$ for some $i \in \ldb n+1, N \rdb$ as otherwise
	\[
		a_n = \sum_{i=n+1}^N \alpha_i (a_i + b_i) = \sum_{i=n+1}^N \alpha_i \frac1{q^i},
	\]
	 which is not possible because $\gcd(q,r) = 1$ and $r \mid \mathsf{d}(a_n)$. Therefore we can set
	 \begin{equation} \label{eq:maximum index m}
	 	m := \max\{i \in \ldb n+1,N \rdb : \alpha_i \neq \beta_i \}.
	 \end{equation}
	 Assume first that $\alpha_m > \beta_m$. Then we can rewrite~\eqref{eq:expression for a_n for case 1} as follows:
	 \begin{equation} \label{eq:modified expression for a_n for case 1}
	 	a_n = (\alpha_m - \beta_m) \frac{q^m r^{\ell_m} - 1}{2q^{2m}r^{\ell_m}} + \sum_{i=m}^N \beta_i \frac1{q^i} + \sum_{i=n+1}^{m-1} \frac{\alpha_i(q^i r^{\ell_i} - 1) + \beta_i(q^i r^{\ell_i} + 1)}{2q^{2i}r^{\ell_i}}.
	 \end{equation}
	After multiplying both sides of~\eqref{eq:modified expression for a_n for case 1} by $2q^{2N}r^{\ell_m}$, one can readily see that each summand involved in the resulting equality except perhaps $q^{2N - 2m}(\alpha_m - \beta_m)(q^m r^{\ell_m} - 1)$ is divisible by $r^{\ell_m - \ell_{m-1}}$. As a consequence, $\alpha_m - \beta_m$ must also be divisible by $r^{\ell_m - \ell_{m-1}}$. Therefore $\alpha_m - \beta_m \ge r^{\ell_m - \ell_{m-1}}$, and so we can use the inequalities $a_m > b_{m+1} > \frac1{2q^{m+1}}$ to obtain that
	\begin{equation} \label{eq:equality contradiction}
		a_n \ge \alpha_m a_m \ge (\alpha_m - \beta_m) b_{m+1} \ge r^{\ell_m - \ell_{m-1}}b_{m+1} > \frac{r^{\ell_m - \ell_{m-1}}}{2q^{m+1}} > 1,
	\end{equation}
	which is a contradiction. In a completely similar manner, we can produce a contradiction after assuming that $\beta_m > \alpha_m$.
	\smallskip
	
	\textsc{Case 2:} $M = \langle A \setminus \{b_n\} \rangle$ for some $n \in \nn$. As we did with $a_n$ in Case~1, in this case we can write $b_n - \alpha_n a_n$ in the following way:
	\begin{equation} \label{eq:expression for a_n for case 2}
		b_n - \alpha a_n = \sum_{i=n+1}^N \alpha_i a_i  + \sum_{i=n+1}^N \beta_i b_i,
	\end{equation}
	for some index $N \in \nn$ with $N > n$ and coefficients $\alpha_{n+1}, \dots, \alpha_N \in \nn_0$ and $\beta_{n+1}, \dots, \beta_N \in \nn_0$ such that either $\alpha_N > 0$ or $\beta_N > 0$. Since $2a_n > b_n$, we see that $\alpha_n \in \{0,1\}$. As in the previous case, we can take the index $m$ as in~\eqref{eq:maximum index m}. First, assume that $\alpha_m > \beta_m$ and then write
		\begin{equation} \label{eq:modified expression for a_n for case 2}
		b_n - \alpha a_n = (\alpha_m - \beta_m) \frac{q^m r^{\ell_m} - 1}{2q^{2m}r^{\ell_m}} + \sum_{i=m}^N \beta_i \frac1{q^i} + \sum_{i=n+1}^{m-1} \frac{\alpha_i(q^i r^{\ell_i} - 1) + \beta_i(q^i r^{\ell_i} + 1)}{2q^{2i}r^{\ell_i}}.
	\end{equation}
	Since $\mathsf{d}(b_n - \alpha_n a_n) \in \{ q^{2n} r^{\ell_n}, 2q^{2n} r^{\ell_n} \}$, after multiplying both sides of~\eqref{eq:modified expression for a_n for case 2} by $2q^{2N} r^{\ell_m}$ we will obtain that $\alpha_m - \beta_m$ is divisible by $r^{\ell_m - \ell_{m-1}}$, and an argument similar to that in~\eqref{eq:equality contradiction} can be used to produced the desired contradiction. In a similar way, we can obtain a contradiction after assuming the inequality $\alpha_m < \beta_m$.
	\smallskip
	
	Having obtained a contradiction in each of the analyzed cases, we can assure now that $A$ is a minimal generating set of $M$, which means that $M$ is an atomic monoid with $\mathcal{A}(M) = A$.
\end{proof}

\medskip
%%%%%%%%%%%%%%%%%%%%%%%%
\subsection{The Ascent of Quasi-Atomicity}

The main purpose of this section is to prove that the property of being quasi-atomic does not ascend to monoid algebras over fields. Before establishing our desired result, which is Theorem~\ref{thm:non-ascent of QA in characteristic p}, we need some basics on irreducibility of certain polynomials over finite fields. We start with a criterion for irreducibility of binomials over finite fields.

\begin{theorem}  \cite[Theorem~3.75]{LN94} \label{thm:irreducibility of binomials}
	Let $q$ be a power of a prime, and let $a \in \ff_q^* $ be an element with order~$e$. Then, for each $t \in \nn$ with $t \ge 2$, the binomial $x^t - a$ is irreducible in $\ff_q[x]$ if and only if $t$ satisfies the following conditions:
	\begin{enumerate}
		\item[(i)] $\gcd\big(t, \frac{q-1}e) = 1$,
		\smallskip
		
		\item[(ii)] each prime factor of $t$ divides $e$, and
		\smallskip
		
		\item[(iii)] if $4 \mid t$, then $4 \mid q-1$.
	\end{enumerate}
\end{theorem}

From the previous theorem, we can deduce the following corollary.

\begin{cor} \label{cor:irreducibility of binomials}
	Let $q$ be a prime such that $q \equiv 1 \pmod{4}$. If $a \in \ff_q^\times$ is a primitive root modulo $q$, then the binomial $x^{(q-1)^n} - a$ is irreducible in $\ff_q[x]$ for every $n \in \nn$.
\end{cor}

\begin{proof}
	Let $a \in \ff_q^\times$ be a primitive root modulo $q$. It suffices to verify the three conditions in the statement of Theorem~\ref{thm:irreducibility of binomials} for $t = (q-1)^n$. First, observe that the order of $a$ in $\ff_q$ is $e = q-1$, so condition~(i) holds. As the prime factors of $(q-1)^n$ and $q-1$ are the same, condition~(ii) also holds. Finally, if $4 \mid t$, then $4 \mid q-1$ follows trivially when $n=1$ while $4 \mid q-1$ follows from $2 \mid q$ when $n \ge 2$. Hence the binomial $x^{(q-1)^n} - a$ is irreducible in $\ff_q[x]$ by virtue of Theorem~\ref{thm:irreducibility of binomials}.
\end{proof}

Let us turn our attention to another criterion for irreducibility of certain trinomials over finite fields. %, given also by Blake et al.

\begin{theorem}  \cite[Theorem~3.3]{BGMVY93} \label{thm:irreducibility of trinomials}
	Let $p$ be a prime such that $p \equiv 3 \pmod{4}$ and write $p+1 = 2^\gamma s$ for some $\gamma, s \in \nn$ with $s$ being odd. Then, for any $k \in \nn$, the trinomial $x^{2^k} - 2ax^{2^{k-1}} - 1$ is irreducible in $\ff_p[x]$ (and so irreducible over $\ff_{p^m}$ for any odd integer $m$), where $a = a_\gamma$ is obtained recursively as follows:
	 \begin{enumerate}
	 	\item[(i)] set $a_1 = 0$;
	 	\smallskip
	 	
	 	\item[(ii)] for $j \in \ldb 2, \gamma-1 \rdb$, set $a_j = \big( \frac{a_{j-1} + 1}2 \big)^{\frac{p+1}4}$;
	 	\smallskip
	 	
	 	\item[(iii)] finally, set $a_\gamma = \big( \frac{a_{\gamma-1} - 1}2 \big)^{\frac{p+1}4}$.
	 \end{enumerate}
\end{theorem}
%
%The following example sheds some light upon the statement of Theorem~\ref{thm:irreducibility of trinomials}.
%
%\begin{example}
%	TODO...
%\end{example}

We are in a position to establish the primary result of this section, namely, that the property of being quasi-atomic does not ascend to monoid algebras over fields in the class of rank-one torsion-free monoids.

\begin{theorem} \label{thm:non-ascent of QA in characteristic p}
	For each $p \in \pp$, there exists a rank-one torsion-free atomic monoid $M$ such that the monoid algebra $\ff_p[M]$ is not quasi-atomic.
\end{theorem}

\begin{proof}
	 First, for each $p \in \pp$, we choose an integer $d \in \nn_{\ge 2}$ that depends on $p$, and then we construct  a pair $(M_d, f_d)$ of a Puiseux monoid $M_d$ and a nonconstant polynomial $f_d(x) \in \ff_p[x]$ with degree at most~$2$ satisfying the following conditions:
	 \begin{enumerate}
	 	\item $M_d$ is atomic, and $\nn_0\big[ \frac1p \big] \subseteq M_d \subseteq \zz\big[ \frac1p, \frac1d \big]$, and
	 	\smallskip
	 	
	 	\item $\deg f_d(x) \in \{1,2\}$ and $f_d\big( x^{d^n} \big)$ is irreducible in $\ff_p[x]$ for every $n \in \nn$.
	 \end{enumerate}
 	\smallskip
 	It is convenient to split the construction of the pairs $(M_d, f_d)$ into the following cases, depending on the residue class of $p$ modulo~$4$.
	\smallskip
	
	\textsc{Case 1:} $p \equiv 2 \pmod{4}$. In this case, $p = 2$, and we set $d := 3$. Now let $M_d$ be the atomic Puiseux monoid $M_{2,3}$ constructed in Proposition~\ref{prop:atomic PMs}, and set $f_d(x) := x^2 + x + 1$. It is known that the polynomial $f_d(x^{3^n}) = x^{2 \cdot 3^n} + x^{3^n} + 1$ is an irreducible polynomial in $\ff_2[x]$ for every $n \in \nn$ (see, for instance, \cite[Lemma~5.3]{CG19}). 
	\smallskip
	
	\textsc{Case 2:} $p \equiv 1 \pmod{4}$. In this case, we set $d := p-1$. Then we let $M_d$ be the atomic Puiseux monoid $M_{p,p-1}$ constructed in Proposition~\ref{prop:atomic PMs}, which contains $\nn_0\big[ \frac1p \big]$ as a submonoid. Now set $f_d(x) = x-a$, where $a \in \ff_p^\times$ is a primitive root modulo~$p$ in $\ff_p$. In light of Corollary~\ref{cor:irreducibility of binomials}, the binomial $x^{(p-1)^n} - a$ is irreducible in $\ff_p[x]$ for every $n \in \nn$. 
	\smallskip
	
	\textsc{Case 3:} $p \equiv 3 \pmod{4}$. In this case, we set $d := 2$. Then we let $M_d$ be the atomic Puiseux monoid $M_{p,2}$ we constructed in Proposition~\ref{prop:atomic PMs}, which contains $\nn_0\big[ \frac1p\big]$ as a submonoid. Also, we set $f_d(x) := x^2 - 2ax - 1$ for some element $a \in \ff_p$ produced as described in the statement of Theorem~\ref{thm:irreducibility of trinomials}. Thus, it follows from the same theorem that the polynomial $f_d\big( x^{2^n}\big)$ is irreducible in $\ff_p[x]$ for every $n \in \nn$ (this implies that $a \neq 0$).
	\smallskip
	
	We proceed to argue that the monoid domain $\ff_p[M_{d}]$ is not even quasi-atomic. Suppose, towards a contradiction, that $\ff_p[M_d]$ is quasi-atomic. In particular, the polynomial $f_d(x)$ is a quasi-atomic element, and so there exist a nonzero $f(x) \in \ff_p[M_d]$ and atoms $a_1(x), \dots, a_\ell(x) \in \ff_p[M_d]$ such that $f(x) f_d(x) = a_1(x) \cdots a_\ell(x)$. In light of condition~(1) above, we can take $n \in \nn$ sufficiently large so that $f\big( x^{(pd)^n}\big)$ and $a_1\big( x^{(pd)^n} \big), \dots, a_\ell\big( x^{(pd)^n} \big)$ are standard polynomials, and so the equality
	\begin{equation}
		f\big( x^{(pd)^n}\big) f_d\big( x^{d^n}\big)^{p^n} = \prod_{i=1}^\ell a_i\big( x^{(pd)^n}\big)
	\end{equation}
	takes place inside the polynomial ring $\ff_p[x]$, which is a UFD. This, along with the fact that $f_d\big( x^{d^n}\big)$ is irreducible in $\ff_p[x]$ (by condition~(2) above), ensures the existence of an exponent $t \in \nn$ and an index $j \in \ldb 1, \ell \rdb$ such that
	\begin{equation} \label{eq:expression for the distinguished atom}
		a_j \big( x^{(pd)^n}\big) = g(x) f_d\big( x^{d^n}\big)^t
	\end{equation}
	for some $g(x) \in \ff_p[x]$ such that $g(x) \mid_{\ff_p[x]} f\big( x^{(pd)^n}\big)$. Now set $G(x) := g(x) f_d\big( x^{d^n}\big)^{t-1} \in \ff_p[x]$, and then take nonzero coefficients $b_1, \dots, b_k \in \ff_p$ and exponents $s_1, \dots, s_k \in \nn_0$ with $s_1 < \cdots < s_k$ such that
	\begin{equation} \label{eq:explicit expression for G}
		G(x) = \sum_{i=1}^k b_i x^{s_i}.
	\end{equation}
	In light of~\eqref{eq:expression for the distinguished atom} and~\eqref{eq:explicit expression for G}, one obtains that $a_j\big( x^{(pd)^n}\big) = f_d\big( x^{d^n}\big) \sum_{i=1}^k b_i x^{s_i}$. Since $\deg f_d(x) \in \{1,2\}$, we can take coefficients $c_0, c_1, c_2 \in \ff_p$ with $c_1 \neq 0$ such that $f_d(x) = c_2x^2 + c_1x + c_0$ (the irreducibility of $f_d(x)$ implies that the coefficient $c_0$ is nonzero, but $c_2 = 0$ when $p \equiv 1 \pmod{4}$). After setting $r_i := \frac{s_i}{(pd)^n}$ for every $i \in \ldb 1,k \rdb$, we can write
	\begin{align} \label{eq:explicit representation of a_j}
		a_j(x) = \big(c_2 x^{\frac2{p^n}} + c_1 x^{\frac1{p^n} } + c_0 \big) \sum_{i=1}^k b_i x^{r_i} = c_2 \sum_{i=1}^k b_i x^{\frac2{p^n} + r_i} + c_1\sum_{i=1}^k b_i x^{\frac1{p^n} + r_i} + c_0 \sum_{i=1}^k b_i x^{r_i}.
	\end{align}
	Before proceeding, we need to argue the following claim.
	\smallskip
	
	\noindent \textsc{Claim.} $r_1, \dots, r_k \in M_d$.
	\smallskip
	
	\noindent \textsc{Proof of Claim.} First, observe that $r_1$ is the order of $a_j(x)$ in $\ff_p[M_d]$, so $r_1 \in M_d$. Then we are done if $k=1$. Assume, therefore, that $k \ge 2$. If $r_2$ belongs to the support of $a_j(x)$, then $r_2 \in M_d$. On the other hand, if $r_2 \notin \text{supp} \, a_j(x)$, then the term $c_0b_2x^{r_2}$ in the rightmost part of~\eqref{eq:explicit representation of a_j} must cancel with either $c_2 b_1x^{\frac2{p^n} + r_1}$ or $c_1 b_1x^{\frac1{p^n} + r_1}$ and, as a result, $r_2 \in \big\{ \frac2{p^n} + r_1, \frac1{p^n} + r_1 \big\} \subseteq M_d$. As $r_1, r_2 \in M_d$, we are done when  $k=2$. Thus, we suppose that $k \ge 3$ and set
	\[
		m := \max \{i \in \ldb 1,k \rdb : r_1, \dots, r_i \in M_d\}.
	\]
	We claim that $m=k$. Assume, by way of contradiction, that $m < k$. Then $r_{m+1} \notin \text{supp} \, a_j(x)$ because of the maximality of $m$, and so there exists $i \in \ldb 1,m \rdb$ such that the term $c_0b_{m+1}x^{r_{m+1}}$ in the rightmost part of~\eqref{eq:explicit representation of a_j} cancels with either $c_2 b_ix^{\frac2{p^n} + r_i}$ or $c_1 b_ix^{\frac1{p^n} + r_i}$. Thus, $r_{m+1} \in \big\{ \frac2{p^n} + r_i, \frac1{p^n} + r_i \big\} \subseteq M_d$, contradicting the maximality of~$m$. Hence the claim is established.
	\smallskip
	
	We can now conclude our proof. In light of the established claim, we see that the polynomial expression $B(x) := \sum_{i=1}^k b_i x^{r_i}$ belongs to $\ff_p[M_d]$. Also, the equality $B(x) = G\big( x^{(pd)^{-n}}\big)$ guarantees that $B(x)$ is not a unit of $\ff_p[M_d]$. Therefore the equality
	\[
		a_j(x) = \big(c_2 x^{\frac2{p^n}} + c_1 x^{\frac1{p^n} } + c_0 \big) \sum_{i=1}^k b_i x^{r_i} = f_d\big( x^{\frac1{p^n}} \big)B(x)
	\]
	contradicts the fact that $a_j(x)$ is irreducible in $\ff_p[M_d]$. Hence we conclude that the monoid domain $\ff_p[M_d]$ is not quasi-atomic.
\end{proof}

As an immediate consequence of Theorem~\ref{thm:non-ascent of QA in characteristic p}, we obtain that none of the properties of almost atomicity and quasi-atomicity ascend to monoid domains over finite fields.

\begin{cor} \label{cor:non-ascent of AA in characteristic p}
	For each $p \in \pp$, there exists a rank-one torsion-free almost atomic (resp., quasi-atomic) monoid $M$ such that the monoid domain $\ff_p[M]$ is not almost atomic (resp., quasi-atomic).
\end{cor}

\bigskip
%%%%%%%%%%%%%%%
%%%%%%%%%%%%%%%
\section*{Acknowledgments}

This paper is the result of a collaboration carried out while the authors were part of CrowdMath 2024, a year-long free online program in mathematical research generously hosted by the MIT Mathematics department and the Art of Problem Solving. The authors are grateful to the advisors and organizers of CrowdMath for making this research opportunity possible. While working on this paper, the second author was kindly supported by the NSF under the award DMS-2213323.

\bigskip
%%%%%%%%%%%%%%%
%%%%%%%%%%%%%%%
\section*{Conflict of Interest Statement}

On behalf of all authors, the corresponding author states that there is no conflict of interest related to this paper.

\bigskip
%%%%%%%%%%%%%%
%%%%%%%%%%%%%%

\end{document}